\documentclass[envcountsect,oribibl,orivec]{llncs}
\usepackage{amssymb}
\usepackage{amsmath}
\usepackage{bussproofs}

\EnableBpAbbreviations
\usepackage{graphics}
 \usepackage{graphicx}
 \usepackage{qtree}
\usepackage{tikz}

\newcommand{\LP}{{\sf LP}}
\newcommand{\FOLP}{{\sf FOLP}}
\newcommand{\M}{{\mathcal M}}
\newcommand{\E}{{\mathcal E}}
\newcommand{\I}{{\mathcal I}}

\newcommand{\D}{{\mathcal D}}
\newcommand{\CS}{{\sf CS}}

\newcommand{\uu}{{\sf u}}
\newcommand{\vv}{{\sf v}}
\newcommand{\ww}{{\sf w}}

 \def\r{\rightarrow}
\def\di{\displaystyle}

\begin{document}
\title{Tableaux for First Order Logic of Proofs}

\author{Meghdad Ghari}
\institute{School of Mathematics,
Institute for Research in Fundamental Sciences (IPM), \\ P.O.Box: 19395-5746, Tehran, Iran \\ \email{ghari@ipm.ir}
}

\maketitle
\begin{abstract}
In this paper we present a tableau proof system for first order logic of proofs $\FOLP$. We show that the tableau system is sound and complete with respect to Mkrtychev models of $\FOLP$. 
\end{abstract}

\section{Introduction}

Artemov in \cite{A1995,A2001} introduced the first propositional justification logic $\LP$, the Logic of Proofs (for more information about justification logics
see \cite{A2008,ArtemovFitting}). Later Artemov and Yavorskaya (Sidon) introdiced in \cite{ArtemovSidon2011} the first order logic of proofs $\FOLP$.  The language of $\FOLP$ extends the language of first order logic by justification terms and expressions of the form $t:_X A$, where $X$ is a set of individual variables. The intended meaning of $t:_X A$ is ``$t$ justifies $A$ in which the variables in $X$ can be substituted for and cannot be quantified.'' Fitting in \cite{Fitting2011,Fitting2014} proposed possible world semantics and Mkrtychev semantics for $\FOLP$.

Various tableau proof systems have been developed for the logic of proofs (see \cite{Finger2010,Fitting2005,Ghari-tableaux-2014,Renne2004,Renne2006}).  The aim of this paper is to present a tableau proof system for $\FOLP$. Our tableau rules are extensions of Renne's tableau rules \cite{Renne2004} for $\LP$. We show that our tableau proof system is sound and complete with respect to Mkrtychev models of $\FOLP$.

\section{The logic $\FOLP$}\label{sec:FOLP}
The language of $\FOLP$ is an
extension of the language of first order logic by expressions 
of the form $t:_X A$, where $A$ is a formula, $t$ is a
justification term and $X$ is a set of individual variables. Following \cite{ArtemovSidon2011} we consider a first order language in which there are no constant symbols, function symbols, and identity, but of course a countable set of individual variables $Var$ (denoted by $x, y, z, \ldots$). 

\textit{Justification terms}  are built up from a countable set of justification variables $JVar$ and a countable set of justification constants $JCons$ by the following grammar:
\[ t::= p~|~c~|~t+t~|~t \cdot t~|~!t~|~gen_x(t),\]

where $p\in JVar$, $c \in JCons$, and $x \in Var$.
$\FOLP$ formulas are constructed from a countable set of predicate symbols of any arity by the following grammar:
\[ A::= Q(x_1,\ldots,x_n)~|~\neg A~|~A\rightarrow A~|~\forall x A~|~\exists x A~|~ t:_X A,\]

where $Q$ is an $n$-place predicate symbol, $t$ is a justification term, and $X \subseteq Var$. 

Free individual variable occurrences in formulas are defined as in the first order logic, with the following addition: the free individual variable occurrences in $t:_X A$ are  the free individual variable occurrences in $A$, provided the variables also occur in $X$, together with all variable occurrences in $X$ itself. The set of all free individual variables of the formula $A$ is denoted by $FVar(A)$. Thus $FVar(t:_X A) = X$. The universal closure of a formula $A$ will be denoted by $\forall A$. 
The notion of substitution of an individual variable for another individual variable is defined as in the first order logic.

If $y$ is an individual variable, then $Xy$ is short for $X \cup \{y\}$, and in addition it means $y \not\in X$.

\begin{definition}\label{def: FOLP axiom system}
Axioms schemes and rules of $\FOLP$ are:\footnote{Ctr and Exp are abbreviations for Contraction and Expansion respectively. The rule AN is called Axiom Necessitation.}

\begin{description}
\item[FOL.] Axiom schemes of first order logic,

\item[Ctr.] $t:_{Xy} A \rightarrow t:_X A$, provided $y\not\in FVar(A)$. 

\item[Exp.] $t:_{X} A \rightarrow t:_{Xy} A$.

\item[Sum.] $s:_X A\rightarrow (s+t):_X A~,~s:_X A\rightarrow (t+s):_X A$.

\item[jK.] $s:_X (A\rightarrow B)\rightarrow(t:_X A\rightarrow (s\cdot t):_X B)$.

\item[jT.] $t:_X A\rightarrow A$.

\item[j4.] $t:_X A\rightarrow !t:_X t:_X A$.

\item[Gen.] $t:_X A \rightarrow gen_x(t):_X \forall x A$, provided $x\not\in X$.

\item[MP.] From $\vdash A$ and $\vdash A \rightarrow B$ infer $\vdash B$.

\item[UG.] From $\vdash A$ infer $\vdash \forall x A$.

\item[AN.]
$\vdash c:A$, where $A$ is an axiom instance and $c$ is an arbitrary justification constant.
\end{description}
\end{definition}

 \begin{definition}\mbox{}
 \begin{enumerate}
\item A \textit{constant specification} $\CS$
for $\FOLP$ is a set of formulas of the form
$c:A$, where $c$ is a
justification constant and $A$ is an axiom instance of $\FOLP$.
\item A constant specification $\CS$ is axiomatically appropriate if  for every axiom instance $A$ there is a justification constant $c$ such that $c:A \in \CS$.
\item Two formulas are variable variants if each can be turned into the other by a  renaming of free and bound individual variables.
\item A constant specification $\CS$ is variant closed provided that whenever $A$ and $B$ are variable variants, $c:A \in \CS$ if and only if $c:B \in\CS$.
\end{enumerate}

\end{definition}

Let $\FOLP_\CS$ be the fragment of $\FOLP$ where the   Axiom Necessitation rule only produces formulas from the given $\CS$.

In the remaining of this section, we recall the definition of Mkrtychev models for $\FOLP$ from \cite{Fitting2014} (Mkrtychev models was first introduced for $\LP$ in \cite{Mkrtychev1997}). First we need the following auxiliary definition. 

\begin{definition}\label{def:K-formulas}
Let $K$ be a non-empty set. 

\begin{enumerate}
\item A $K$-formula is the result of substituting some free individual variables in an $\FOLP$ formula with members of $K$. 

\item A $K$-formula is closed if it contains no free occurrences of individual variables. 

\item For a $K$-formula $A$, let $K(A)$ be the set of all members of $K$ that occur in $A$. 

\item For a formula $F(\vec{x})$ and $\vec{a}\in K$, by  $F(\vec{a})$ we mean that all free occurrences of the individual variables in $\vec{x}$ have been replaced with corresponding occurrences in $\vec{a}$. This is sometimes denoted by $F \{\vec{x} / \vec{a} \}$.
\end{enumerate}

\end{definition}


\begin{definition}\label{def:FOLP-models}
A Mkrtychev model $\M=(\D,\I,\E)$ for $\FOLP_\CS$ (or $\FOLP_\CS$-model, for short) consists of:
\begin{itemize}
\item A non-empty set $\D$, called the domain of the model. The definitions of (closed) $\D$-formulas and $\D(A)$, for a $\D$-formula $A$, are similar to Definition \ref{def:K-formulas}, where $K$ is replaced by $\D$.

\item The interpretation $\I$ assigns to each $n$-place predicate symbol some $n$-ary relation on $\D$.

\item The admissible evidence function $\E$ assigns to each justification term a set of $\D$-formulas meeting the following conditions:

 \begin{description}
 \item[$\E 1.$]  $c:A\in\CS$ implies $A\in\E(c)$.
 
 \item[$\E 2.$]  $A\r B\in\E(s)$ and $A\in\E(t)$ implies $B\in\E(s\cdot t)$.
 
  \item[$\E 3.$]  $\E(s)\cup \E(t)\subseteq\E(s+t)$.
  
   \item[$\E 4.$] $A\in\E(t)$ implies $t:_X A\in\E(!t)$, where $\D(A) \subseteq X \subseteq \D$.
  
 \item[$\E 5.$] $A\in \E(t)$ implies $\forall x A \in \E(gen_x(t))$.
 
 \item[$\E 6.$]  $a \in \D$ and $A(x) \in \E(t)$ implies $A(a) \in \E(t)$.
  \end{description}
  \end{itemize}
  \end{definition}

Condition $\E6$ is called the Instantiation Condition in \cite{Fitting2014}.

\begin{definition}\label{def:forcing relation}
For an $\FOLP_\CS$-model $\M=(\D,\I,\E)$ and a closed $\D$-formula we define when the formula is true in $\M$ as follows:
\begin{enumerate}
\item $\M\Vdash Q(\vec{a})$ if{f} $\vec{a} \in \I(Q)$, for $n$-place predicate symbol $Q$ and $\vec{a} \in \D$.
 \item $\M\Vdash \neg A$ if{f} $\M\not\Vdash A$.
  \item $\M\Vdash A\r B$ if{f} $\M\not\Vdash A$ or $\M\Vdash B$.
  \item $\M\Vdash \forall x A(x)$ if{f} $\M \Vdash A(a)$ for every $a \in \D$.
   \item $\M\Vdash \exists x A(x)$ if{f} $\M \Vdash A(a)$ for some $a \in \D$.
 \item $\M\Vdash t:_X A$ if{f} $A\in\E(t)$ and $\M \Vdash \forall A$.
\end{enumerate}
If $\M\Vdash F$ then it is said that $F$ is true in $\M$ or $\M$ satisfies $F$.
\end{definition}

A sentence $F$ is $\FOLP_\CS$-valid if it is true in every $\FOLP_\CS$-model. For a set $S$ of sentences, $\M\Vdash S$ provided that $\M\Vdash F$ for all formulas $F$ in $S$. Note that given a constant specification $\CS$ for $\FOLP$, and a model $\M$ of $\FOLP_\CS$ we have $\M\Vdash \CS$ (in this case it is said that $\M$ respects $\CS$).

The proof of soundness and completeness theorems of $\FOLP$ are given in \cite{Fitting2014}.

\begin{theorem}\label{Sound Compl JL}
Let  $\CS$ be an axiomatically appropriate and variant closed constant specification for $\FOLP$. Then a sentence $F$ is provable in  $\FOLP_\CS$ if{f} $F$ is $\FOLP_\CS$-valid.
\end{theorem}

\section{Tableaux}

Tableau proof systems for the logic of proofs are given in \cite{Fitting2005, Renne2004, Renne2006}. In this section we extend them and present  tableaux for $\FOLP$.

Let  $Par$ be a denumerable set of new individual variables, i.e. $Par \cap Var = \emptyset$. The members of $Par$ are called \textit{parameters}, with typical members denoted  $\uu, \vv, \ww$. Parameters are never quantified. The definitions of (closed) $Par$-formulas, $Par$-instance of a formula, and $Par(A)$, for a $Par$-formula $A$, are similar to Definition \ref{def:K-formulas}, where $K$ is replaced by $Par$. Notice that closed $Par$-formulas may contain free parameters but do not contain free individual variables. 

Tableau proofs will be of sentences of $\FOLP$ but will use closed $Par$-formulas. An $\FOLP_\CS$-tableau for a sentence is a binary tree labeled by closed $Par$-formulas with the negation of that sentence at the root constructed by applying $\FOLP$ tableau rules from Table \ref{table:tableau rules FOLP}.  An $\FOLP_\CS$-tableau branch closes if one of the following holds:

\begin{enumerate}
\item Both $A$ and $\neg A$ occurs in the branch, for some closed $Par$-formula $A$.

\item $\neg c: A$ occurs in the branch, where $c:A\in\CS$.

\end{enumerate}

A tableau closes if all branches of the tableau close. An $\FOLP_\CS$-tableau proof for a sentence $F$ is a closed tableau beginning with $\neg F$ (the root of the tableau) using only $\FOLP$ tableau rules. An $\FOLP_\CS$-tableau for a finite set $S$ of closed $Par$-formulas begins with a single branch whose nodes consist of the formulas of $S$ as roots.

\begin{table}[ht]
\centering\renewcommand{\arraystretch}{1.2}
\begin{tabular}{|lc|}
\hline
~First order logic rules: &\\
\hline

\multicolumn{2}{|c|}{
\AXC{}\noLine 
\UIC{$\neg\neg A$}\RightLabel{$(F\neg)$}
\UIC{$A$}\noLine
\UIC{}
\DP
}\\\hline

\AXC{$A\rightarrow B$}\RightLabel{$(T\r)$}
\UIC{$\neg A | B$}
\DP
&
\AXC{}\noLine 
\UIC{$\neg(A\rightarrow B)$}\RightLabel{$(F\r)$}
\UIC{$A$}\noLine
\UIC{$\neg B$}\noLine
\UIC{}
\DisplayProof

\\ \hline
\AXC{}\noLine 
\UIC{$\forall x A(x)$}\RightLabel{$(T\forall)$}
\UIC{$A(\uu)$}\noLine
\UIC{}
\DisplayProof
&\AXC{$\neg \exists x A(x)$}\RightLabel{$(F\exists)$}
\UIC{$\neg A(\uu)$}
\DP
\\
~$\uu$ is any parameter
&
~$\uu$ is any parameter
\\\hline

\AXC{}\noLine 
\UIC{$\exists x A(x)$}\RightLabel{$(T\exists)$}
\UIC{$A(\uu)$}\noLine
\UIC{}
\DisplayProof
&\AXC{$\neg\forall x A(x)$}\RightLabel{$(F\forall)$}
\UIC{$\neg A(\uu)$}
\DP
\\
~$\uu$ is a new parameter
&
$\uu$ is a new parameter
\\ \hline
~Justification logic rules:&\\\hline
 \AXC{$t:_X A$} \RightLabel{$(T:)$}
 \UIC{$\forall A$}
 \DP
 &
 \AXC{}\noLine 
\UIC{$\neg t+s:_X A$}\RightLabel{$(F+)$}
\UIC{$\neg t:_X A$}\noLine
\UIC{$\neg s:_X A$}\noLine
\UIC{}
\DisplayProof
\\
\hline
\AXC{}\noLine 
\UIC{$\neg s\cdot t:_X B$}
\RightLabel{$(F\cdot)$}
\UIC{$\neg s:_X (A\rightarrow B) | \neg t:_X A$}\noLine
\UIC{$Par(A)\subseteq X$}\noLine
\UIC{}
\DisplayProof

&\AXC{}\noLine 
\UIC{$\neg !t:_X t:_X A$} \RightLabel{$(F!)$}
 \UIC{$\neg t:_X A$}
 \noLine
\UIC{}
 \DP
\\ \hline
\AXC{}\noLine
\UIC{$\neg t:_{X} A$} \RightLabel{$(Ctr)$}
 \UIC{$\neg t:_{X\uu} A$}
 \noLine
\UIC{}
 \DP
 &
 \AXC{}\noLine 
\UIC{$\neg t:_{X\uu} A$} \RightLabel{$(Exp)$} 
\UIC{$\neg t:_{X} A$}\noLine
 \UIC{$\uu\not\in Par(A)$}\noLine
\UIC{}
\DP
\\\hline
\AXC{}\noLine
\UIC{$\neg t:_X A(\uu)$} \RightLabel{$(Ins)$}
 \UIC{$\neg t:_X A(x)$}\noLine
 \UIC{}
 \DP
 &
 \AXC{$\neg gen_x(t):_X \forall x A $} \RightLabel{$(gen_x)$} 
\UIC{$\neg t:_X A$}
\DP
\\
\hline
\multicolumn{2}{|l|}{In all justification logic rules $X \subseteq Par$.}
\\
\hline
\end{tabular}\vspace{0.3cm}
 \caption{Tableau rules for $\FOLP$.}\label{table:tableau rules FOLP}
\end{table}

\begin{example} 
We give an $\FOLP_\CS$-tableau proof of the sentence 

\[p: \forall x A(x) \r \forall x (c \cdot p):_{\{x\}} A(x),\]
 where $p \in JVar$ and $\CS$ contains $c:(\forall x A(x)\r  A(x))$. An axiomatic proof of this sentence is given in \cite{ArtemovSidon2011}. This sentence is an explicit counterpart of the Converse Barcan Formula $\Box \forall x A(x) \r \forall x \Box A(x)$.
 
\vspace*{0.2cm}
\Tree [.$1.~\neg(p:\forall xA(x)\r \forall x(c\cdot p):_{\{x\}}A(x))$ 
[.$2.~p:\forall xA(x)$ 
[.$3.~\neg \forall x(c\cdot p):_{\{x\}}A(x)$ 
[.$4.~\neg (c\cdot p):_{\{\uu\}}A(\uu)$     [.$5.~\neg c:_{\{\uu\}}(\forall xA(x)\r A(\uu))$ 
[.$7.~\neg c:_{\{u\}}(\forall xA(x)\r A(x))$ {$9.~\neg c:(\forall xA(x)\r A(x))$ \\ $\otimes$} ] ] 
!\qsetw{5cm} 
[.$6.~\neg p:_{\{\uu\}}\forall xA(x)$ {$8.~\neg p:\forall xA(x)$ \\ $\otimes$} ] ]  ] ]  ]
\vspace*{0.2cm}

Formulas 2 and 3 are from 1 by rule $(F\r)$, 4 is from 3 by rule $(F\forall)$, where $\uu$ is a new parameter, 5 and 6 are from 4  by rule $(F\cdot)$, 7 is from 5 by rule $(Ins)$, and 8 and 9 are from 6 and 7, respectively, by rule $(Exp)$. 
\end{example}

Let us now show the soundness of $\FOLP$ tableau system.

\begin{definition}\label{def:Par-satisfiable}
A $Par$-formula $A(\uu_1, \ldots, \uu_n)$, where $\uu_1, \ldots, \uu_n$ are all parameters of $A$, is satisfiable in an $\FOLP_\CS$-model $\M = (\D, \I, \E)$, denoted by \linebreak $\M \Vdash A(\uu_1, \ldots, \uu_n)$, if $\M \Vdash A(a_1, \ldots, a_n)$ for some $a_1, \ldots, a_n \in \D$. A tableau branch is satisfiable in a model $\M$ if every formula of the branch is satisfiable.  
\end{definition}

\begin{lemma}\label{lem: soundness lemma}
Let $\pi$ be any branch of an $\FOLP_\CS$-tableau and $\M$ be an $\FOLP_\CS$-model that satisfies all the formulas occur in $\pi$. If an $\FOLP$ tableau rule is applied to $\pi$, then it produces at least one extension $\pi'$ such that $\M$ satisfies all the formulas occur in $\pi'$.
\end{lemma}
\begin{proof}
Suppose that a tableau branch $\pi$ is satisfiable in the model $\M=(\D,\I,\E)$, and $\pi'$ is obtained by applying an $\FOLP$ tableau rule to $\pi$. To prove the lemma, we consider each rule in turn. The cases for the propositional logic rules are standard. Hence, we need consider only the rules for quantifiers and $\FOLP$ rules. 

Suppose that the rule $(T\forall)$ is applied
\[
\AXC{$\forall x A(x, \vec{\ww})$}\RightLabel{$(T\forall)$}
\UIC{$A(\uu, \vec{\ww})$}
\DisplayProof
\]
where $\uu,\vec{\ww} \in Par$. Since $\M \Vdash \forall x A(x, \vec{\ww})$, we have $\M \Vdash  \forall x A(x, \vec{b})$ for some $\vec{b} \in \D$. Thus, $\M \Vdash  A(a, \vec{b})$ for every $a \in \D$. Then, obviously $\M \Vdash  A(\uu, \vec{\ww})$. Hence $\M \Vdash \pi'$ as desired. The case of the rule $(F\exists)$ is similar.

Suppose that  the rule $(T\exists)$ is applied
\[
\AXC{$\exists x A(x, \vec{\ww})$}\RightLabel{$(T\exists)$}
\UIC{$A(\uu, \vec{\ww})$}
\DisplayProof
\]
where $\vec{\ww} \in Par$ and $\uu$ is a new parameter in $\pi$. Since $\M \Vdash \exists x A(x, \vec{\ww})$, we have $\M \Vdash  \exists x A(x,\vec{b})$  for some $\vec{b} \in \D$. Thus, $\M \Vdash  A(a, \vec{b})$  for some $a \in \D$. Then, obviously $\M \Vdash  A(\uu, \vec{\ww})$. Hence $\M' \Vdash \pi'$ as desired. The case of the rule $(T\forall)$ is similar.

Suppose that the rule $(T:)$ is applied
\[
\AXC{$t:_{\{\vec{\ww}, \vec{\vv}\}} A(\vec{\ww},\vec{x})$} \RightLabel{$(T:)$} 
\UIC{$\forall \vec{x} A(\vec{\ww},\vec{x})$}
\DP
\]
Since $\M \Vdash t:_{\{\vec{\ww}, \vec{\vv}\}} A(\vec{\ww},\vec{x})$, we have $\M \Vdash t:_{\{\vec{a}, \vec{b}\}} A(\vec{a},\vec{x})$ for some $\vec{a}, \vec{b} \in \D$. Thus $\M \Vdash \forall \vec{x} A(\vec{a},\vec{x})$. Then, obviously $\M \Vdash \forall \vec{x} A(\vec{\ww},\vec{x})$. Hence $\M \Vdash \pi'$ as desired. 

Suppose that  the rule $(F+)$ is applied
\[
\AXC{$\neg t+s:_{\{\vec{\ww}, \vec{\vv}\}} A(\vec{\ww},\vec{x})$}\RightLabel{$(F+)$} 
\UIC{$\neg t:_{\{\vec{\ww}, \vec{\vv}\}} A(\vec{\ww},\vec{x})$}\noLine
\UIC{$\neg s:_{\{\vec{\ww}, \vec{\vv}\}} A(\vec{\ww},\vec{x})$}
\DisplayProof
\]
Since $\M  \Vdash \neg  t+s:_{\{\vec{\ww}, \vec{\vv}\}} A(\vec{\ww},\vec{x})$, we have $\M  \Vdash \neg  t+s:_{\{\vec{a}, \vec{b}\}} A(\vec{a},\vec{x})$ for some $\vec{a}, \vec{b} \in \D$. Thus either $A(\vec{a},\vec{x}) \not\in \E(t+s)$ or $\M \not\Vdash \forall \vec{x} A(\vec{a},\vec{x})$. In the former case we have $A(\vec{a},\vec{x}) \not\in \E(t) \cup \E(s)$, and hence $\M  \not\Vdash  t :_{\{\vec{a}, \vec{b}\}} A(\vec{a},\vec{x})$ and $\M \not \Vdash  s :_{\{\vec{a}, \vec{b}\}} A(\vec{a},\vec{x})$. We get the same results in the latter case. In either case $\M  \Vdash \neg t :_{\{\vec{\ww}, \vec{\vv}\}} A(\vec{\ww},\vec{x})$ and $\M  \Vdash \neg s :_{\{\vec{\ww}, \vec{\vv}\}} A(\vec{\ww},\vec{x})$. Hence $\M \Vdash \pi'$ as desired. 

Suppose that the rule $(F\cdot)$ is applied
\[
\AXC{$\neg s\cdot t:_{\{\vec{\ww}, \vec{\vv}\}} B(\vec{\ww},\vec{x})$} 
\RightLabel{$(F\cdot)$}
\UIC{$\neg s:_{\{\vec{\ww}, \vec{\vv}\}} (A(\vec{\ww'},\vec{\vv'},\vec{y})\rightarrow B(\vec{\ww},\vec{x})) | \neg t:_{\{\vec{\ww}, \vec{\vv}\}} A(\vec{\ww'},\vec{\vv'},\vec{y})$}
\DisplayProof
\]
where $\{\vec{\ww'} \} \subseteq \{ \vec{\ww} \}$ and $\{\vec{\vv'} \} \subseteq \{ \vec{\vv} \}$. Since $\M \Vdash \neg s\cdot t:_{\{\vec{\ww}, \vec{\vv}\}} B(\vec{\ww},\vec{x})$, we have $\M \Vdash \neg s\cdot t:_{\{\vec{a}, \vec{b}\}} B(\vec{a},\vec{x})$ for some $\vec{a}, \vec{b} \in \D$.  Thus either $B(\vec{a},\vec{x}) \not\in \E(s\cdot t)$ or $\M \not\Vdash \forall \vec{x} B(\vec{a},\vec{x})$. In the former case we have either $A(\vec{a'},\vec{b'},\vec{y}) \rightarrow B(\vec{a},\vec{x}) \not\in \E(s)$ or $A(\vec{a'},\vec{b'},\vec{y}) \not\in \E(t)$, where $A(\vec{a'},\vec{b'},\vec{y}) = A(\vec{\ww'},\vec{\vv'},\vec{y}) \{ \vec{\ww}/\vec{a} , \vec{\vv}/\vec{b}\}$, and hence either $\M \not\Vdash s :_{\{\vec{a}, \vec{b}\}} (A(\vec{a'},\vec{b'},\vec{y}) \r B(\vec{a},\vec{x}))$ or $\M \not\Vdash t :_{\{\vec{a}, \vec{b}\}} A(\vec{a'},\vec{b'},\vec{y})$. In the latter case, either $\M \not\Vdash \forall A$ and hence $\M\not\Vdash  t :_{\{\vec{a}, \vec{b}\}} A(\vec{a'},\vec{b'},\vec{y})$, or $\M \Vdash \forall A$ and hence $\M \not\Vdash s :_{\{\vec{a}, \vec{b}\}} (A(\vec{a'},\vec{b'},\vec{y}) \r B(\vec{a},\vec{x}))$, since $\M  \not\Vdash  \forall (A \r B)$. Thus, in both cases we have either $\M  \Vdash \neg s :_{\{\vec{a}, \vec{b}\}} (A(\vec{a'},\vec{b'},\vec{y}) \r B(\vec{a},\vec{x}))$ or $\M \Vdash \neg t :_{\{\vec{a}, \vec{b}\}} A(\vec{a'},\vec{b'},\vec{y})$. Therefore either $\M \Vdash \neg s :_{\{\vec{\ww}, \vec{\vv}\}} (A(\vec{\ww'},\vec{\vv'},\vec{y}) \r B(\vec{\ww},\vec{x}))$ or  $\M \Vdash \neg t :_{\{\vec{\ww}, \vec{\vv}\}} A(\vec{\ww'},\vec{\vv'},\vec{y})$.

Suppose that the rule $(F!)$ is applied
\[
\AXC{$\neg !t:_{\{\vec{\ww}, \vec{\vv}\}} t:_{\{\vec{\ww}, \vec{\vv}\}} A(\vec{\ww},\vec{x})$} \RightLabel{$(F!)$} 
\UIC{$\neg t:_{\{\vec{\ww}, \vec{\vv}\}} A(\vec{\ww},\vec{x})$}
\DP
\]
Since $\M \Vdash \neg !t:_{\{\vec{\ww}, \vec{\vv}\}} t:_{\{\vec{\ww}, \vec{\vv}\}} A(\vec{\ww},\vec{x})$, we have $\M  \Vdash \neg !t:_{\{\vec{a}, \vec{b}\}} t:_{\{\vec{a}, \vec{b}\}} A(\vec{a},\vec{x})$ for some $\vec{a}, \vec{b} \in \D$. Thus either  $t:_{\{\vec{a}, \vec{b}\}} A(\vec{a},\vec{x}) \not\in \E(!t)$ or $\M \not\Vdash t :_{\{\vec{a}, \vec{b}\}} A(\vec{a},\vec{x})$. In the former case we have $A(\vec{a},\vec{x}) \not\in \E(t)$, and hence $\M \Vdash \neg t :_{\{\vec{a}, \vec{b}\}} A(\vec{a},\vec{x})$. In either case $\M \Vdash \neg t :_{\{\vec{\ww}, \vec{\vv}\}} A(\vec{\ww},\vec{x})$. Hence $\M \Vdash \pi'$ as desired. 

Suppose that the rule $(Ctr)$ is applied
\[
\AXC{$\neg t:_{\{\vec{\ww}, \vec{\vv}\}} A(\vec{\ww},\vec{x})$} \RightLabel{$(Ctr)$} 
\UIC{$\neg t:_{\{\vec{\ww}, \vec{\vv},\uu\}} A(\vec{\ww},\vec{x})$}
\DP
\]
  Since $\M \Vdash \neg t:_{\{\vec{\ww}, \vec{\vv}\}} A(\vec{\ww},\vec{x})$, we have $\M \Vdash \neg t:_{\{\vec{a}, \vec{b}\}} A(\vec{a},\vec{x})$ for some $\vec{a}, \vec{b} \in \D$. Thus either  $A(\vec{a},\vec{x}) \not\in \E(t)$ or $\M \not\Vdash \forall A$. From this it follows that $\M  \Vdash \neg t:_{\{\vec{a}, \vec{b},d\}} A(\vec{a},\vec{x})$ for an arbitrary $d \in \D$. Therefore $\M \Vdash \neg t:_{\{\vec{\ww}, \vec{\vv},\uu\}} A(\vec{\ww},\vec{x})$. Hence $\M \Vdash \pi'$ as desired. 

Suppose that the rule $(Exp)$ is applied
\[
\AXC{$\neg t:_{\{\vec{\ww}, \vec{\vv},\uu\}} A(\vec{\ww},\vec{x})$} \RightLabel{$(Exp)$} 
\UIC{$\neg t:_{\{\vec{\ww}, \vec{\vv}\}} A(\vec{\ww},\vec{x})$}
\DP
\]
where $\uu \not\in Par(A)$. Since $\M \Vdash \neg t:_{\{\vec{\ww}, \vec{\vv},\uu\}} A(\vec{\ww},\vec{x})$, we have $\M \Vdash \neg t:_{\{\vec{a}, \vec{b},d\}} A(\vec{a},\vec{x})$ for some $\vec{a}, \vec{b},d \in \D$. Thus either $A(\vec{a},\vec{x}) \not\in \E(t)$ or $\M \not\Vdash \forall A$. From this it follows that $\M  \Vdash \neg t:_{\{\vec{a}, \vec{b}\}} A(\vec{a},\vec{x})$. Therefore $\M \Vdash \neg t:_{\{\vec{\ww}, \vec{\vv}\}} A(\vec{\ww},\vec{x})$. Hence $\M \Vdash \pi'$ as desired. 

Suppose that the rule $(Ins)$ is applied
\[
\AXC{$\neg t:_{\{\vec{\ww}, \vec{\vv},\uu\}} A(\vec{\ww},\vec{y},\uu)$} \RightLabel{$(Ins)$} 
\UIC{$\neg t:_{\{\vec{\ww}, \vec{\vv},\uu\}} A(\vec{\ww},\vec{y},x)$}
\DP
\]
 Since $\M \Vdash \neg t:_{\{\vec{\ww}, \vec{\vv},\uu\}} A(\vec{\ww},\vec{y},\uu)$, we have $\M \Vdash \neg t:_{\{\vec{a}, \vec{b},d\}} A(\vec{a},\vec{y},d)$ for some $\vec{a}, \vec{b},d \in \D$. Thus either $A(\vec{a},\vec{y},d) \not\in \E(t)$ or $\M \not\Vdash \forall \vec{y} A(\vec{a},\vec{y},d)$. Thus, by the Instantiation Condition $(\E6)$,  either $A(\vec{a},\vec{y},x) \not\in \E(t)$ or $\M \not\Vdash \forall x\forall \vec{y} A(\vec{a},\vec{y},x)$. Thus $\M \Vdash \neg t:_{\{\vec{a}, \vec{b},d\}} A(\vec{a},\vec{y},x)$. Therefore $\M \Vdash \neg t:_{\{\vec{\ww}, \vec{\vv},\uu\}} A(\vec{\ww},\vec{y},x)$. Hence $\M \Vdash \pi'$ as desired.

Suppose that the rule $(gen_x)$ is applied
\[
\AXC{$\neg gen_x(t):_{\{\vec{\ww}, \vec{\vv}\}} \forall x A$} \RightLabel{$(gen_x)$} 
\UIC{$\neg t:_{\{\vec{\ww}, \vec{\vv}\}} A$}
\DP
\]
We consider the case where $A=A(\vec{\ww},\vec{y},x)$, i.e. $x \in FVar(A)$. The case that $x$ is not free in $A$ is treated similarly.  Since $\M \Vdash \neg  gen_x(t):_{\{\vec{\ww}, \vec{\vv}\}} \forall x A(\vec{\ww},\vec{y},x)$, we have $\M \Vdash \neg  gen_x(t):_{\{\vec{a}, \vec{b}\}} \forall x A(\vec{a},\vec{y},x)$ for some $\vec{a}, \vec{b} \in \D$. Thus either $\forall x A(\vec{a},\vec{y},x) \not\in \E(gen_x(t))$ or $\M \not\Vdash \forall \vec{y} \forall x A(\vec{a},\vec{y},x)$. Hence either $A(\vec{a},\vec{y},x) \not\in \E(t)$ or $\M \not\Vdash \forall  A$. In either case $\M  \Vdash \neg t:_{\{\vec{a}, \vec{b}\}}   A(\vec{a},\vec{y},x)$, and therefore $\M  \Vdash \neg  t:_{\{\vec{\ww}, \vec{\vv}\}}   A(\vec{\ww},\vec{y},x)$. Hence $\M \Vdash \pi'$ as desired. 
\qed
\end{proof}

\begin{theorem}[Soundness]
Let $A$ be a sentence of $\FOLP$. If $A$ has an $\FOLP_\CS$-tableau proof, then it is $\FOLP_\CS$-valid.
\end{theorem}
\begin{proof}
If the sentence $A$ is not $\FOLP_\CS$-valid, then there is an $\FOLP_\CS$-model $\M$ such that $\M\Vdash \neg A$. Then, by Lemma \ref{lem: soundness lemma}, there is no closed $\FOLP_\CS$-tableau beginning with $\neg A$. Therefore, $A$ does not have an $\FOLP_\CS$-tableau proof.\qed
\end{proof}

Next we will prove the completeness theorem by making use of maximal consistent sets. 

\begin{definition}
Suppose $\Gamma$ is a set of closed $Par$-formulas. 

\begin{enumerate}
\item $\Gamma$ is tableau $\FOLP_\CS$-consistent if there is no closed $\FOLP_\CS$-tableau beginning with any finite subset of $\Gamma$. 

\item $\Gamma$ is maximal if it has no proper tableau consistent extension (w.r.t. closed $Par$-formulas).

\item $\Gamma$ is $E$-complete (with members of $Par$ as witnesses) if 

\begin{description}

\item[$\bullet$] $\exists x A(x) \in \Gamma$ implies $A(\uu)\in \Gamma$ for some $\uu\in Par$.

\item[$\bullet$]  $\neg \forall x A(x) \in \Gamma$ implies $\neg A(\uu)\in \Gamma$ for some $\uu\in Par$.

\end{description}

\end{enumerate}
\end{definition}

By making use of the Henkin construction it is not hard to show the following result.

\begin{lemma}
Every tableau $\FOLP_\CS$-consistent set of sentences of $\FOLP$ can be extended to a tableau $\FOLP_\CS$-consistent, maximal and $E$-complete set of closed $Par$-formulas. 
\end{lemma}

It is easy to show that $E$-complete maximally tableau $\FOLP_\CS$-consistent sets are closed under $\FOLP_\CS$-tableau rules. For a non-branching rule like
\begin{prooftree}
\AXC{$\alpha$}
\UIC{$\alpha_1$}\noLine
\UIC{$\alpha_2$}
\end{prooftree}
this means that if $\alpha$ is in a $E$-complete maximally tableau $\FOLP_\CS$-consistent set $\Gamma$, then both $\alpha_1\in\Gamma$ and $\alpha_2\in\Gamma$. For a branching rule like

\[\di{\frac{\beta}{\beta_1 | \beta_2}}\]
this means that if $\beta$ is in a $E$-complete maximally tableau $\FOLP_\CS$-consistent set $\Gamma$, then $\beta_1\in\Gamma$ or $\beta_2\in\Gamma$. For the rule $(F\cdot)$ this means that if $\neg s \cdot t :_X B\in\Gamma$, then for every formula $A$ such that $Par(A) \subseteq X$ either $\neg s:_X (A\r B) \in\Gamma$ or $\neg t:_X A \in\Gamma$.

\begin{lemma}\label{lem:downward saturated}
Suppose $\Gamma$ is an $E$-complete maximally tableau $\FOLP_\CS$-consistent set of closed $Par$-formulas. Then $\Gamma$ is closed under $\FOLP_\CS$-tableau rules.
\end{lemma}

\begin{proof}
The proof for rules $(F\neg)$, $(F\r)$, and $(T\r)$, are standard. We detail the proof for other tableau rules.

\begin{description}
\item[$(T\forall)$] Suppose $\forall x A \in \Gamma$ and $\uu$ is an arbitrary parameter. We want to show that $A(\uu)\in \Gamma$. If this is not the case, since $\Gamma$ is maximal, then $\Gamma \cup \{A(\uu)\}$ is not tableau $\FOLP_\CS$-consistent. Hence there is a closed  $\FOLP_\CS$-tableau for a finite subset, say $\Gamma_0 \cup \{A(\uu)\}$. But $\Gamma_0 \cup \{\forall x A\}$ is a finite subset of $\Gamma$ and, using rule $(T\forall)$, there is a closed  $\FOLP_\CS$-tableau for it, contra the tableau consistency of $\Gamma$. The case of $(F\exists)$ is similar.

\item[$(T\exists)$] Suppose $\exists x A \in \Gamma$. Since $\Gamma$ is $E$-complete, $A(\uu) \in \Gamma$ for some parameter $\uu$. The case of $(F\forall)$ is similar.

\item[$(F\cdot)$] Suppose $\neg s \cdot t :_X B\in\Gamma$. Suppose towards a contradiction that for some formula $A$ such that $Par(A) \subseteq X$ we have $\neg s:_X (A\r B)\not\in\Gamma$ and $\neg t:_X A\not\in\Gamma$. Since $\Gamma$ is maximal,  $\Gamma \cup \{ \neg s:_X (A\r B) \}$ and $\Gamma \cup \{ \neg t:_X A\}$ are not tableau consistent.  Thus there are closed tableaux for finite subsets, say $\Gamma_1 \cup \{ \neg s:_X (A\r B) \}$ and $\Gamma_2 \cup \{ \neg t:_X A\}$. But $\Gamma_1 \cup \Gamma_2 \cup \{\neg  s \cdot t :_X  B \}$ is a finite subset of $\Gamma$ and, using rule $(F\cdot)$, there is a closed  $\FOLP_\CS$-tableau for it, contra the tableau consistency of $\Gamma$.

\item[$(F!)$] Suppose $\neg !t:_X t:_{X} A \in \Gamma$. We want to show that $\neg t:_{X} A \in \Gamma$. If it is not the case, since $\Gamma$ is maximal, then $\Gamma \cup \{\neg t:_{X} A\}$ is not tableau $\FOLP_\CS$-consistent. Hence there is a closed  $\FOLP_\CS$-tableau for a finite subset, say $\Gamma_0 \cup \{\neg t:_{X} A\}$. But $\Gamma_0 \cup \{\neg !t:_X t:_{X} A\}$ is a finite subset of $\Gamma$ and, using rule $(F!)$, there is a closed  $\FOLP_\CS$-tableau for it, contra the tableau consistency of $\Gamma$. The cases of rules $(T:)$ and $(F+)$ are similar.

\item[$(Ctr)$] Suppose $\neg t:_{X} A \in \Gamma$. We want to show that $\neg t:_{X\uu} A \in \Gamma$. If it is not the case, since $\Gamma$ is maximal, then $\Gamma \cup \{\neg t:_{X\uu} A\}$ is not tableau $\FOLP_\CS$-consistent. Hence there is a closed  $\FOLP_\CS$-tableau for a finite subset, say $\Gamma_0 \cup \{\neg t:_{X\uu} A\}$. But $\Gamma_0 \cup \{\neg t:_{X} A\}$ is a finite subset of $\Gamma$ and, using rule $(Ctr)$, there is a closed  $\FOLP_\CS$-tableau for it, contra the tableau consistency of $\Gamma$.

\item[$(Exp)$] Suppose $\neg t:_{X\uu} A \in \Gamma$ and $\uu \not\in Par(A)$. We want to show that $\neg t:_{X} A \in \Gamma$. If it is not the case, since $\Gamma$ is maximal, then $\Gamma \cup \{\neg t:_{X} A\}$ is not tableau $\FOLP_\CS$-consistent. Hence there is a closed  $\FOLP_\CS$-tableau for a finite subset, say $\Gamma_0 \cup \{\neg t:_{X} A\}$. But $\Gamma_0 \cup \{\neg t:_{X\uu} A\}$ is a finite subset of $\Gamma$ and, using rule $(Exp)$, there is a closed  $\FOLP_\CS$-tableau for it, contra the tableau consistency of $\Gamma$.

\item[$(Ins)$] Suppose $\neg t:_{X} A(\uu)\in \Gamma$. We want to show that $\neg t:_{X} A(x) \in \Gamma$. If it is not the case, since $\Gamma$ is maximal, then $\Gamma \cup \{\neg t:_{X} A(x)\}$ is not tableau $\FOLP_\CS$-consistent. Hence there is a closed  $\FOLP_\CS$-tableau for a finite subset, say $\Gamma_0 \cup \{\neg t:_{X} A(x)\}$. But $\Gamma_0 \cup \{\neg t:_{X} A(\uu)\}$ is a finite subset of $\Gamma$ and, using rule $(Ins)$, there is a closed  $\FOLP_\CS$-tableau for it, contra the tableau consistency of $\Gamma$.

\item[$(gen_x)$] Suppose $\neg gen_x(t):_{X} \forall x A \in \Gamma$. We want to show that $\neg t:_{X} A \in \Gamma$. If it is not the case, since $\Gamma$ is maximal, then $\Gamma \cup \{\neg t:_{X} A\}$ is not tableau $\FOLP_\CS$-consistent. Hence there is a closed  $\FOLP_\CS$-tableau for a finite subset, say $\Gamma_0 \cup \{\neg t:_{X} A\}$. But $\Gamma_0 \cup \{\neg gen_x(t):_{X} \forall x A\}$ is a finite subset of $\Gamma$ and, using rule $(gen_x)$, there is a closed  $\FOLP_\CS$-tableau for it, contra the tableau consistency of $\Gamma$. \qed
\end{description}

\end{proof}

\begin{definition}\label{def:canonical model tableau}
Given an $E$-complete maximally tableau $\FOLP_\CS$-consistent set $\Gamma$ of closed $Par$-formulas, the canonical model $\M=(\D,\I,\E)$ with respect to $\Gamma$ and $\CS$ is defined as follows:

\begin{itemize}
\item $\D=Par$.

\item $\I(Q) = \{ (\uu_1, \ldots, \uu_n) \in \D~|~Q(\uu_1, \ldots, \uu_n) \in \Gamma \}$, for any n-place relation symbol $Q$.

\item $\E(t) = \{A~|~\neg t:_{Par(A)} A \not\in \Gamma\}$.

\end{itemize}
\end{definition}

\begin{lemma}
Given an $E$-complete maximally tableau $\FOLP_\CS$-consistent set $\Gamma$ of closed $Par$-formulas, the canonical model $\M=(\D,\I,\E)$ with respect to $\Gamma$ and $\CS$ is an $\FOLP_\CS$-model.
\end{lemma}

\begin{proof}
Suppose $\Gamma$ is an $E$-complete maximally tableau $\FOLP_\CS$-consistent set, and $\M=(\D,\I,\E)$ is the canonical model  with respect to $\Gamma$ and $\CS$. We  will show that the admissible evidence function $\E$ satisfies $\E 1$-$\E 6$ from Definition \ref{def:FOLP-models}.

\begin{description}
\item[$(\E 1)$] Suppose that $c:A\in\CS$. Then $Par(A) = \emptyset$. We have to show that $A\in\E(c)$. Since $\Gamma$ is tableau $\FOLP_\CS$-consistent, $\neg c:A\not\in\Gamma$ and hence $\neg c:_{Par(A)} A\not\in\Gamma$. Thus  $A\in\E(c)$.

\item[$(\E 2)$] Suppose that $A\in\E(t)$ and $A\r B\in\E(s)$. We have to show that \linebreak $B\in\E(s\cdot t)$. Let $X=Par(A \r B)=Par(A) \cup Par(B)$. By the definition of $\E$,   $\neg t:_{Par(A)} A\not\in\Gamma$ and $\neg s:_X (A\r B)\not \in\Gamma$. Since $\Gamma$ is closed under rule $(Exp)$, $\neg t:_X A \not\in\Gamma$. Since $\Gamma$ is closed under rule $(F\cdot)$, $\neg s\cdot t:_X B\not\in\Gamma$. Since $\Gamma$ is closed under rule $(Ctr)$, $\neg s\cdot t:_{Par(B)} B\not\in\Gamma$. Hence, by the definition of $\E$, $B\in\E(s\cdot t)$.

\item[$(\E 3)$] Suppose that $A\in\E(s)\cup \E(t)$. We have to show that $A\in\E(s+t)$. If $A\in\E(s)$, then $\neg s:_{Par(A)} A \not\in \Gamma$. Since $\Gamma$ is closed under rule  $(F+)$,  $\neg s+t:_{Par(A)} A\not\in\Gamma$. Therefore, $A\in\E(s+t)$. The case that $A\in\E(t)$ is similar.

\item[$(\E 4)$] Suppose that $A\in\E(t)$ and $\D(A)=Par(A) \subseteq X$. First consider the case that $X \neq Par(A)$. We have to show that $t:_X A \in \E(!t)$. By the  definition of $\E$, $\neg t:_{Par(A)} A \not\in \Gamma$. Since $\Gamma$ is closed under rule $(Exp)$, $\neg t:_X A \not\in \Gamma$. Since $\Gamma$ is closed under rule $(F!)$,  $\neg !t:_X t:_X A\not\in\Gamma$. Therefore,  $t:_X A\in\E(!t)$. The case that $X = Par(A)$ is similar.

\item[$(\E 5)$] Suppose that $A\in \E(t)$. We have to show that $\forall x A\in\E(gen_x(t))$. By the definition of $\E$, $\neg t:_{Par(A)} A \not\in \Gamma$. Since $\Gamma$ is closed under rule  $(gen_x)$,  $\neg gen_x(t) :_{Par(A)} \forall x A\not\in\Gamma$. Therefore, $\forall x A\in\E(gen_x(t))$.

\item[$(\E 6)$] Suppose that $A(x) \in\E(t)$ and $\uu \in \D=Par$. We have to show that \linebreak $A(\uu) \in \E(t)$. Let $X = Par(A(x))$. By the definition of $\E$, $\neg t:_{X} A(x) \not\in \Gamma$. We distinguish two cases. (1) Suppose $\uu \not\in X$. Since $\Gamma$ is closed under rule $(Exp)$, $\neg t:_{X\uu} A(x) \not\in \Gamma$. Since $\Gamma$ is closed under rule $(Ins)$, $\neg t:_{X\uu} A(\uu) \not\in \Gamma$. Therefore,  $A(\uu) \in \E(t)$. (2) Suppose $\uu \in X$. Since $\Gamma$ is closed under rule $(Ins)$, $\neg t:_{X} A(\uu) \not\in \Gamma$. Therefore,  $A(\uu) \in \E(t)$. \qed
\end{description}

\end{proof}

\begin{lemma}[Truth Lemma] 
Suppose $\Gamma$ is an $E$-complete maximally tableau $\FOLP_\CS$-consistent set of closed $Par$-formulas and $\M=(\D,\I,\E)$ is the canonical model with respect to $\Gamma$ and $\CS$. Then for every closed $Par$-formula $F$:

\begin{enumerate}
\item $F \in \Gamma$ implies $\M \Vdash F$.
\item $\neg F \in \Gamma$ implies $\M \not\Vdash F$.
\end{enumerate}
\end{lemma}

\begin{proof}
By induction on the complexity of $F$. The base case and the propositional and quantified inductive cases are standard. The proof for the case that $F=t:_X A$ is as follows. Note that $Par(A) \subseteq X$.

Assume $t:_X A \in\Gamma$. Since $\Gamma$ is $\FOLP_\CS$-consistent, $\neg t:_X A \not\in \Gamma$. If $X=Par(A)$, then $A \in \E(t)$. If $X \neq Par(A)$, then since $\Gamma$ is closed under rule $(Ctr)$, $\neg t:_{Par(A)} A \not\in \Gamma$. Thus $A \in \E(t)$. On the other hand, since $t:_X A \in\Gamma$ and $\Gamma$ is closed under rule $(T:)$, $\forall A \in \Gamma$. Let $FVar(A) = \{\vec{x}\}$. Thus $\forall \vec{x} A(\vec{x}) \in \Gamma$. Since $\Gamma$ is closed under rule $(T\forall)$, $A(\vec{\uu}) \in \Gamma$ for any $\vec{\uu} \in Par$. Hence, by the induction hypothesis, $\M \Vdash A(\vec{\uu})$ for any $\vec{\uu} \in Par=\D$, and thus $\M \Vdash \forall A$. Therefore, $\M \Vdash t:_X A$.

Assume $\neg t:_X A\in\Gamma$. If $X=Par(A)$, then $A \not\in \E(t)$. On the other hand, if $X \neq Par(A)$, then since $\Gamma$ is closed under rule $(Exp)$, $\neg t:_{Par(A)} A \in \Gamma$, and hence  $A \not\in \E(t)$. In either cases $\M \not\Vdash t:_X A$. 
\qed
\end{proof}

\begin{theorem}[Completeness]\label{thm:completeness tableaux}
Let $A$ be a sentence of $\FOLP$. If $A$ is $\FOLP_\CS$-valid, then it has a $\FOLP_\CS$-tableau proof.
\end{theorem}
\begin{proof}
If $A$ does not have a $\FOLP_\CS$-tableau proof, then $\{\neg A\}$ is a $\FOLP_\CS$-consistent set and can be extended to a tableau $\FOLP_\CS$-consistent, maximal and $E$-complete set $\Gamma$ of closed $Par$-formulas. Since $\neg A\in\Gamma$, by the Truth Lemma, $\M\not\Vdash A$, where $\M$ is the canonical model of $\FOLP_\CS$ with respect to $\Gamma$ and $\CS$. Therefore $A$ is not $\FOLP_\CS$-valid.\qed
\end{proof}

\noindent
{\bf Acknowledgments}\\

This research was in part supported by a grant from IPM. (No. 95030416)


\end{document}